\documentclass[]{amsproc}
 \usepackage[]{amsmath, amsthm, amsfonts, amssymb}
 \usepackage[all]{xy}

\usepackage[mathscr]{eucal}
 \usepackage{hyperref}
  \usepackage[usenames,dvipsnames]{color}

\xyoption{curve}
\usepackage{graphicx}
\usepackage{tabularx}

 \newcommand {\C} {{\mathbb C}}
 \newcommand{\BC}{{\mathbb C}}
\newcommand{\BQ}{{\mathbb Q}}

 \newcommand{\BZ}{{\mathbb Z}}
 \newcommand {\Q} {{\mathbb Q}}

 \newcommand {\PP} {{\mathbb P}}
 \newcommand {\bH}{{\mathbb{H}}}

 \newcommand {\dt} {{\bullet}}
\newcommand{\scE}{\mathscr{E}}
\newcommand{\sF} {\mathscr{F}}

 \newcommand {\gr} {\text{\rm Gr}}
 
 \newcommand {\mO}{\mathcal{O}}
 
 \newcommand{\im}{\text{\rm im}}
 \newcommand {\cl}{\text{\rm cl}}
 \newcommand {\ch}{\text{\rm ch}} 

\newcommand{\ds}{\displaystyle}

\newcommand{\Ch}{\text{\rm CH}}
\newcommand{\ol}{\overline}

\newcommand{\Pic}{\text{\rm Pic}}

\newcommand{\m}{{\mathcal M}}

 \newtheorem{cor}[subsection]{Corollary}
 \newtheorem{lemma}[subsection]{Lemma}
 \newtheorem{prop}[subsection]{Proposition}
 
 \newtheorem{rmk}[subsection]{Remark}

  \numberwithin{equation}{section}

\newcommand{\propref}[1]{Proposition \ref{#1}}
\newcommand{\lemref}[1]{Lemma \ref{#1}}
\newcommand{\coref}[1]{Corollary \ref{#1}}

\begin{document}
\title{The smooth center of the  cohomology of  a singular variety}
\author{Donu Arapura}
\address{Department of Mathematics, Purdue University, West Lafayette, IN 47906, U.S.A.}
\email{dvb@math.purdue.edu}
\thanks{The first author was partially supported by the NSF}
\author{Xi Chen}
\address{Department of Mathematical Sciences, University of Alberta, Edmonton, AB T6G 2G1, Canada}
\email{xic@ualberta.ca}
\thanks{The second author was partially supported by a grant from the Natural Sciences
and Engineering Research Council of Canada.}
\author{Su-Jeong Kang}
\address{Department of Mathematics, Providence College, Providence, RI 02918, U.S.A.}
\email{skang2@providence.edu}
\thanks{The third author was partially supported by the Summer Scholar Award from Providence College.}

\date{}

\begin{abstract}
We study constraints on the Chern classes of a vector bundle on a singular variety.  We use this constraint to study a variety which carries a Hodge cycle that are not a linear combination of Chern classes of vector bundles on it. 
\end{abstract}

\maketitle

As is well known, the Hodge conjecture is equivalent to the statement that
Hodge cycles on a smooth projective variety are rational linear combinations of
Chern classes of algebraic vector bundles  (see \cite{ArapuraKang} for further
explanation). This is no longer true for singular varieties.
We will refer to a projective variety $X$ as {\em
  strange} if $X$ carries a weight $2p$ Hodge cycle in $H^{2p}(X,\Q)$ for some $p$,  which is not a  linear combination
of Chern classes. Examples of strange varieties have been constructed by  Bloch \cite[appendix A]{Jannsen},
Barbieri-Viale and Srinivas \cite{bs}, and two of the authors
\cite{ArapuraKang}.
In attempting to  understand the precise nature of strangeness,
 we were led to the following construction:
The {\em smooth center} of the cohomology of a complex projective variety is
the sum of the pullbacks of cohomologies of smooth varieties dominated by it. More generally,
suppose that $G$ is a contravariant functor from the category of
algebraic varieties over some field to a suitable abelian category. Given a projective variety $X$,
 we define the smooth center of  $G(X)$  by
$$
G_{\rm sm} (X) = \sum_{(Y,f) \in \mathcal{C}(X)} \,   f^*G(Y)
\subseteq G(X),
$$
where $\mathcal{C}(X)$ is the collection of pairs  $(Y,f)$ consisting
of a nonsingular variety $Y$ and a morphism $f:X\to Y$. 
 It is  clear that $G_{\rm sm}$ is a subfunctor of $G$.  
While we hope that the above construction is interesting for itself, we
focus on the motivating problem.
There are two cases of interest to us: when $G=K^0$ is the
Grothendieck group of algebraic vector bundles, and when 
the ground field is $\C$ and $G=H^*$ singular
rational cohomology regarded as either a vector space or a mixed Hodge
structure. We will see that $K^0_{\rm sm}(X)=K^0(X)$ is always true,
but that $H^*_{\rm sm}(X)\not=H^*(X)$ in general. It will follow 
that Chern classes lie in $H^*_{\rm sm}(X)$ and that this gives a
genuine constraint. In the latter part of this paper, we examine some
new and previously known examples where we can exhibit strangeness using this method.

Perhaps we should add that,
 as the referee has pointed out to us,  a device similar to the smooth center was employed by Fulton \cite{fulton} for defining Chow cohomology of singular varieties.

\section{Definition and properties of the smooth center of the cohomology}

 As above, we define the smooth center of the cohomology of a complex projective variety $X$ by
$$
H^i_{\rm sm} (X,\Q) = \sum_{(Y,f) \in \mathcal{C}(X)} \,   f^*H^i(Y,\Q).
$$
Since the Hodge structures $H^i(Y,\Q)$ are pure, we obtain:

\begin{lemma}
  $H^i_{\rm sm} (X,\Q)$ is  pure of weight $i$, and therefore $W_{i-1}H^i(X,\Q) \cap H^i_{\rm sm} (X,\Q)=0$.
\end{lemma}

From the lemma, we can identify  $H^i_{\rm sm} (X,\Q)$ as a sub-Hodge
structure of $\gr^W_{i}H^i(X,\Q)$. In \cite{ArapuraKang}, we defined 
a natural filtration $F^{\dt}_{\rm DR}$ on $H^i(X,\C)$ called the de Rham
filtration
\begin{equation}\label{eq:dr-filtration}
F^p_{\rm DR} H^i(X,\C) = H^i(X,\C) \cap \rho \left( \im[\bH^i(X,\Omega^{\geq p}_X) \to \bH^i(X,\Omega^{\dt}_X)] \right)
\end{equation}
where $\rho: \Omega^{\dt}_X \to \tilde{\Omega}^{\dt}_X$ is the morphism from the complex of sheaves of K\"{a}hler differentials $\Omega^{\dt}_X$ on $X$ \cite[\S 16.6]{ega4} to the Du Bois complex $\tilde{\Omega}^{\dt}_X$ \cite{dubois}. Since $\rho$ preserves filtrations on complexes, this implies that the de Rham filtration is generally finer than the Hodge filtration.

\begin{lemma}\label{lemma:contained-in-dr}
$\ds{F^p H^i_{\rm sm}(X,\C) \cap H^i_{\rm sm}(X,\Q)  \subseteq F^p_{\rm DR}H^i(X,\C) \cap \gr^W_i H^i(X,\Q)}$,
where $F^{\dt}$ is the Hodge filtration on $H^i_{\rm sm}(X,\Q)$.
\end{lemma}

\begin{proof}
Let $\alpha = f^*(\beta) \in F^pH^i_{\rm sm}(X,\C) $ where $(Y,f) \in \mathcal{C}(X)$. Since a morphism of Hodge structures preserves the Hodge filtration strictly, $\beta$ should be a cycle in $F^p H^i(Y,\C) =F^p_{\rm DR}H^i(Y,\C)$ and hence $\alpha  \in f^*(F^p_{\rm DR}H^i(Y,\C)) \subseteq  F^p_{\rm DR} H^i(X,\C)$ by the functoriality of the de Rham filtration. 
\end{proof}

\begin{lemma}
$H^*_{\rm sm}(X,\Q)$ is a subring of $H^*(X,\Q)$ under cup product.
\end{lemma}
\begin{proof}
This  follows immediately from
 $$
 \alpha \cup \beta = \Delta^*(\alpha \otimes \beta) = \Delta^* (f \times g)^*(\alpha_1 \otimes \beta_1) \in H^{i+j}_{\rm sm}(X,\Q)
 $$ 
for any $\alpha =f^*(\alpha_1) \in H^i_{\rm sm}(X,\Q)$ and $\beta  = g^*(\beta_1) \in H^j_{\rm sm}(X,\Q)$ where $f: X \to T$ and $g: X \to S$ are morphisms to smooth varieties $T$ and $S$.
\end{proof}

\begin{lemma}\label{lemma:K0sm}
$K^0_{\rm sm}(X)= K^0(X)$.  
\end{lemma}

\begin{proof}
  Recall that for a vector bundle $\scE_i$ on $X$, there exists a smooth
  variety $M_i$, an embedding $\iota : X \to M_i$ and a vector bundle $\sF_i$ on
  $M_i$ such that $\scE_i = \iota^* \sF_i$ (cf. \cite{ArapuraKang}). By taking
  products, we can see that any pair of vector bundles $\scE_i$ are
  the pullbacks  of a pair of vector bundles under an embedding of $X$ into  common smooth variety
  $M=M_1\times M_2$. Since any
  element of $K^0(X)$ is a difference of vector bundles, the lemma follows.
\end{proof}

Given a natural transformation of contravariant functors $G\to G'$, it
is clear that we get a natural transformation $G_{\rm sm}\to G'_{\rm sm}$.

\begin{lemma}\label{lemma:cp-in-smooth-coh}
For any vector bundle $\scE$ on $X$,  $c_p(\scE) \in H^{2p}_{\rm
  sm}(X,\Q)\cap F^p H^{2p}_{\rm sm}(X,\C) $. 
\end{lemma}
\begin{proof}
Let $\mathcal{H}^{2p}(X) =  H^{2p}(X,\Q)\cap F^p H^{2p}(X,\C) $
denote the space of Hodge cycles, and let $\mathcal{H}_{\rm sm}^{2p}(X) \subseteq H^{2p}_{\rm
  sm}(X,\Q)\cap F^p H^{2p}_{\rm sm}(X,\C) $
denote its smooth center.
$\scE$ can be regarded as an element of $K^0(X)$, which equals $K^0_{\rm
  sm}(X)$ by the previous
lemma. By the above remark, $c_p$ gives a natural transformation
$c_p:K^0_{\rm sm}(X)\to \mathcal{H}^{2p}_{\rm sm}(X)$. 
\end{proof}

\begin{rmk}
A similar argument shows that the image of the $\ell$-adic Chern class map $K^0(X)\otimes \Q_\ell\to H^{2p}(X_{et},\Q)$ lies in $H_{\rm sm}^{2p}(X_{et},\Q)$ when $X$ is defined over an arbitrary algebraically closed field.
\end{rmk}

\begin{cor}\label{cor:cp-in-smooth-coh}
$\im [c_p: K^0(X) \otimes \Q \to H^{2p}(X,\Q)] \subseteq H^{2p}_{\rm sm}(X,\Q) \cap F^p H^{2p}_{\rm sm}(X,\C)$.
If the Hodge conjecture holds in degree $2p$ for all smooth varieties, then the equality holds.
\end{cor}
\begin{proof}
It is enough to show the last statement. Let $\alpha =f^*(\beta)  \in H^{2p}_{\rm sm}(X,\Q) \cap F^pH^{2p}_{\rm sm}(X,\C)$ where $(Y, f) \in \mathcal{C}(X)$. Then $\beta$ is a Hodge $(p,p)$-cycle on $Y$, and the Hodge conjecture for $Y$ implies that $\beta = \sum_i c_p(\scE_i)$ where $\scE_i$'s are vector bundles on $Y$.  Hence, $\alpha =  f^*(\beta) = \sum_i c_p\left( f^*(\scE_i) \right)  \in \im \left[ c_p : K^0(X) \otimes \Q \to H^{2p}(X,\Q)\right]$.
\end{proof}

\begin{cor}\label{cor:c1toH2sm}
$c_1: K^0(X) \otimes \Q \to F^1H^2_{\rm sm}(X,\C) \cap H^2_{\rm sm}(X,\Q)$ is a surjection. 
\end{cor}
\begin{proof}
It follows immediately from Lemma \ref{lemma:contained-in-dr} and Corollary \ref{cor:cp-in-smooth-coh}, since $\im(c_1) = F^1_{\rm DR}H^2(X,\C) \cap \gr^W_2 H^2(X,\Q)$ by \cite[Lemma 3.14]{ArapuraKang}.
\end{proof}

By putting Lemma \ref{lemma:cp-in-smooth-coh} and Corollary \ref{cor:cp-in-smooth-coh} together, we can recover the one of the main
results of \cite{ArapuraKang} that the image of $c_p$ lies in
$F_{\rm DR}^p H^{2p}(X)\cap H^{2p}(X,\Q)$.

\section{$H_{\text{sm}}^1(C)$ of a singular curve}\label{SEC000INTRO}

In this section, we calculate $H_{\text{sm}}^1(C,\Q)$, when $C$ is a curve.
For every smooth projective variety $Y$, let  $\alpha: Y\to A(Y)$
denote the Albanese map, then it is well known that
\begin{equation}\label{E000750}
H^1(Y,\BQ) = \alpha^* H^1(A(Y), \BQ).
\end{equation}
Therefore, we conclude
\begin{equation}\label{E000749}
H_{\text{sm}}^1(X, \BQ) = \sum_{f} f^* H^1(A, \BQ)
\end{equation}
where $f: X\to A$ runs over all morphisms $f$ from $X$ to an abelian variety $A$.

Now let $C$ be a singular curve and $\nu: \Gamma\to C$ be its normalization. We observe that
for every map $f: C\to A$ from $C$ to an abelian variety $A$, the corresponding
$f\circ \nu:\Gamma\to A$ has to factor through the Albanese $A(\Gamma)$, i.e., the Jacobian
$J(\Gamma)$ of $\Gamma$. If $J(\Gamma)$ is a simple abelian variety, either the map is constant or $J(\Gamma)$ is isogenous onto its image. Since $J(\Gamma)$ is simple for a very general curve $\Gamma$,
this line of argument leads to the conclusion that $H_{\text{sm}}^1(C,\Q)
= 0$ for such a curve.
More precisely, we have

\begin{prop}\label{PROP000800}
Let $C$ be a singular, integral and projective curve with normalization $\nu: \Gamma\to C$.
Suppose that the Jacobian $J(\Gamma)$ of $\Gamma$ is a simple abelian variety.
If $H_{{\rm sm}}^1(C, \BQ) \ne 0$, then
\begin{enumerate}
\item[(a)]
$\nu: \Gamma \to C$ is an immersion at every point $p\in \Gamma$ in
the sense that 
the map $\nu^* \Omega_C^1\to \Omega_{\Gamma}^1$ on K\"ahler differentials
is surjective at every point $p\in \Gamma$;
\item[(b)] $p - q$ is torsion in $\Pic(\Gamma)$ for all $p, q\in \Gamma$ satisfying
$\nu(p) = \nu(q)$.
\end{enumerate}
In particular, $H_{{\rm sm}}^1(C, \BQ) = 0$ for a singular curve $C$ whose normalization
is a very general curve of genus $\ge 3$.
\end{prop}

\begin{proof}
By \eqref{E000749}, $f^*H^1(A, \BQ) \ne 0$ for some morphism
$f: C\to A$ from $C$ to an abelian variety $A$.
By the universality of the Jacobian, we have the commutative diagram
\begin{equation}\label{E000705}
\xymatrix{
\Gamma \ar[r]^\nu \ar[d]^{\alpha} & C\ar[r]^f & A\\
J(\Gamma) \ar[urr]_{\rho}
}
\end{equation}
Since $J(\Gamma)$ is simple, $\rho$ maps $J(\Gamma)$ to a point or 
onto an abelian variety $B\subset A$ isogenously.
If it is the former, then $f$ is constant and $f^*H^1(A, \BQ) = 0$. 
Since $f^*H^1(A, \BQ) \ne 0$, we must have the latter.
Then $f(C)\subset B$ and we have the diagram
\begin{equation}\label{E000748}
\xymatrix{
\Gamma \ar[r]^\nu \ar[d]^{\alpha} & C\ar[r]^f & B \ar[r]^{\rho^\vee} & J(\Gamma)\\
J(\Gamma) \ar[urr]^{\rho} \ar[urrr]_{\phi}
}
\end{equation}
where $\rho: J(\Gamma)\to B$ is an isogeny, $\rho^\vee: B\to J(\Gamma)$ is its
dual and hence $\phi = \rho^\vee\circ \rho = [n]$ is multiplication by a
nonzero integer $n$.

Since $\alpha: \Gamma\hookrightarrow
J(\Gamma)$ is an embedding and $\rho: J(\Gamma)\to B$ is \'etale, it follows that
$\rho\circ \alpha: \Gamma\to B$ is an immersion. Consequently, $\nu: \Gamma\to C$
is an immersion.

For two points $p$ and $q$ on $\Gamma$ satisfying $\nu(p) = \nu(q)$,
we have 
\begin{equation}\label{E000700}
\phi(\alpha(p)) = \phi(\alpha(q)) \Leftrightarrow n(p-q) = 0 
\text{ in }
\Pic(\Gamma)
\end{equation}
i.e., $p-q$ is torsion in $\Pic(\Gamma)$.

For a very general curve $\Gamma$, $J(\Gamma)$ is simple. We claim that there do not exist
$p\ne q$ on $\Gamma$ such that $p - q$ is torsion if $g = g(\Gamma)\ge 3$. Otherwise,
$n(p - q) = 0$ in $\Pic(\Gamma)$ for some $n\in \BZ^+$. Then we have a map $j: \Gamma\to \PP^1$
of degree $n$ totally ramified at $p$ and $q$. Clearly, $n > 1$ and $j$ has at most
$2g + 2$ ramification points. A dimension count shows such $\Gamma$ lies in a subvariety of
dimension $2g-1$ in
the moduli space $\m_g$ of genus $g$ curves. On the other hand, $\dim \m_g = 3g - 3 > 2g - 1$.
\end{proof}

Therefore, to have nontrivial $H_{\text{sm}}^1(C,\Q)$, we need either 
a non-simple Jacobian $J(\Gamma)$ or 
torsion classes $p - q$ for all $p\ne q$ on $\Gamma$ over a singular point of $C$.
Note that $J(\Gamma)$ will fail to be simple if we have a finite map $\Gamma\to Y$
from $\Gamma$ to a smooth projective curve $Y$ with $g(\Gamma) > g(Y) > 0$. This leads us
to consider singular curves $C$ admitting a finite map $C\to Y$ to a
smooth curve $Y$. Then we obtain a map $\Gamma\to Y$, where $\Gamma$
is the normalization. 
We can see that $G$ is simple if the map
$f:\Gamma\to Y$ is sufficiently general, where $G$ is the connect component of $\ker [J(\Gamma)\to J(Y)]$ containing the identity.

\begin{prop}\label{PROP801}
Let $\varphi: C \to Y$ be a finite map from an integral projective curve $C$ to a smooth curve $Y$.
Let $\nu: \Gamma\to C$ be the normalization of $C$ and let
\begin{equation}\label{E000791}
J(\Gamma) \xrightarrow{\phi \times g_1\times g_2\times ...\times g_m} J(Y)\times G_1\times G_2\times ... \times G_m
\end{equation}
be an isogeny, where $\phi = (\varphi\circ \nu)_*$ and
$G_1, ..., G_m$ are simple abelian varieties such that $G = G_1\times G_2\times ...\times G_m$ is isogenous to
the connected component of the kernel of $J(\Gamma)\to J(Y)$ containing the identity.
Suppose that there is a node $q\in C$ with $\nu^{-1}(q) = \{p_1, p_2\}$ such that
$g_k(p_1- p_2)$ is non-torsion in $G_k$ for all $k=1,2,...,m$.
Then
\begin{equation}\label{E000704}
H_{{\rm sm}}^1(C, \BQ) = \varphi^* H^1(Y, \BQ).
\end{equation}
\end{prop}

\begin{proof}
It suffices to prove that
\begin{equation}\label{E000780}
f^* H^1(A, \BQ) \subset \varphi^* H^1(Y, \BQ)
\end{equation}
for all maps $f: C\to A$ from $C$ to an abelian variety $A$. For such $f: C\to A$,
we again have the diagram
\eqref{E000705}. Combining it with the isogeny between $J(\Gamma)$ and $J(Y)\times G$, we have
\begin{equation}\label{E000790}
\xymatrix{
\Gamma \ar[r]^\nu \ar[d]_{\alpha} & C\ar[r]^f & A\\
J(\Gamma) \ar[urr]^{\rho}\ar[d]_{\phi\times g}\\
J(Y)\times G\ar[uurr]_{\varepsilon\times\eta}
}
\end{equation}
where $g = g_1\times g_2 \times ...\times g_m$,
$\varepsilon = \rho\circ \phi^\vee$ and $\eta = \eta_1\times\eta_2\times ...\times \eta_m = \rho\circ g^\vee$
with $\eta_k = g_k^\vee$ for $k=1,2,...,m$.

Clearly, \eqref{E000780} holds if $\eta$ is constant. Suppose that $\eta$ is non-constant. Then
$\eta_k$ is non-constant for some $1\le k\le m$. Since $G_k$ is simple, $\eta_k$ is an isogeny between
$G_k$ and $\eta_k(G_k)$.

Since $\nu(p_1) = \nu(p_2)$, $\eta_k\circ g_k\circ \alpha(p_1) = \eta_k\circ g_k\circ \alpha(p_2)$.
Using the same argument as in the proof of \propref{PROP000800}, we see that $g_k(\alpha(p_1)) - g_k(\alpha(p_2))$
is torsion in $G$.
Contradiction. Therefore, $\eta$ must be constant and \eqref{E000780} holds.
\end{proof}

\begin{cor}\label{COR800}
Let $Y$ be a smooth irreducible projective curve. For a very general nodal curve $C$ that is finite over $Y$
with a map $\varphi: C\to Y$, \eqref{E000704} holds.
\end{cor}

\begin{proof}
Using the same notations as in \propref{PROP801}, we let $\nu: \Gamma\to C$ be the normalization of $C$.
Since $C$ is very general, $\Gamma$ is very general. Hence the connected component of $\ker[J(\Gamma)\to J(Y)]$ containing the
identity is simple and we have an isogeny $\phi\times g: J(\Gamma)\to
J(Y)\times G$ with $G$ a simple abelian variety.

Let $q$ be a node of $C$ and $\nu^{-1}(q) = \{ p_1, p_2\}$. Since $C$ is very general, $p_1 - p_2$ is non-torsion
in $J(\Gamma)$. It follows that $(\phi\times g)(p_1 - p_2)$ is non-torsion in $J(Y)\times G$. On the other hand,
$\varphi(\nu(p_1)) = \varphi(\nu(p_2))$ and hence $\phi(p_1 - p_2) = 0$. Therefore, $g(p_1 - p_2)$ is non-torsion
in $G$. Then it follows from \propref{PROP801} that \eqref{E000704} holds.
\end{proof}

\begin{cor}\label{COR801}
Let $f: X\to Y$ be a finite map between two smooth irreducible projective curves. 
Suppose that $Y$ is maximum in the sense that there is no finite map $X\to T$ from $X$ to a smooth projective curve $T$
satisfying $g(Y) < g(T) < g(X)$. Let $Z$ be a nodal curve together with
\begin{equation}\label{E000703}
f = h\circ \nu: X\xrightarrow{\nu} Z \xrightarrow{h} Y
\end{equation}
where $\nu$ is the normalization of $Z$. If $Z$ has a node $q$ such that $h(q)$ is a very general point on $Y$, then
$H_{{\rm sm}}^1(Z, \BQ) = h^* H^1(Y, \BQ)$.
\end{cor}

\begin{proof}
We let
\begin{equation}\label{E000792}
J(X) \xrightarrow{\phi\times g_1\times g_2\times ...  \times g_m} J(Y) \times G_1\times G_2 \times ...\times G_m
\end{equation}
be an isogeny with $\phi = f_*$ and $G_1, G_2, ..., G_m$ simple abelian varieties such that
$G = G_1\times G_2 \times ... \times G_m$ is isogenous to the connected component of $\ker[J(X) \to J(Y)]$ containing the identity.
By \propref{PROP801}, it suffices to show that $g_k(p_1 - p_2)$ is non-torsion in $G_k$ for all $k=1,2,...,m$,
where $\nu^{-1}(q) = \{p_1, p_2\}$.

Now let us consider the set
\begin{equation}\label{E000793}
\Sigma = \{ \phi\times \eta_k \mid  \eta_k = \lambda_k \circ g_k, \, \lambda_k\in \mathop{\text{End}}(G_k), \,  1\le k\le m\}.
\end{equation}
Clearly, $\Sigma$ is a countable set. For $\sigma \in \Sigma$, $\sigma$ is a surjective map $J(X)\to J(Y)\times G_k$
factoring through $\phi\times g_k$. Let $C = \sigma(\alpha(X))$ be the reduced image of $\alpha(X)$ under $\sigma$,
where $\alpha$ is the Jacobian embedding $X\to J(X)$. Since $\alpha(X)$ generates $J(X)$, $C$ generates $J(Y)\times G_k$.
Hence $g(C) \ge \dim J(Y) + \dim G_k > g(Y)$. So we have a finite map $X\to T$ with $T$ the normalization of $C$.
By our hypothesis on $Y$, we must have $T\cong X$. That is, $\sigma\circ \alpha: X\to C$ is the normalization of $C$.
Then $\sigma(\alpha(p_1)) = \sigma(\alpha(p_2))$ if and only if $\sigma\circ \alpha$ maps $p_1$ and $p_2$ to the singular locus
$C_{\text{sing}}$ of $C$. So we let 
\begin{equation}\label{E000794}
\Delta = \{ p\in X \mid  \sigma(\alpha(p))\in C_{\text{sing}} \text{ for some } \sigma\in \Sigma, \,  C = \sigma(\alpha(X)) \}.
\end{equation}
Again, $\Delta$ is a countable set of points on $X$. Therefore, $p_1,p_2\not\in \Delta$ as we assume
$h(q) = f(p_1) = f(p_2)$ to be a very general point on $Y$.

Note that $g_k(p_1 - p_2)$ is torsion in $G_k$ for some $k$ if and only if $\sigma(p_1 - p_2) = 0$ for some $\sigma\in \Sigma$,
while this only happens when $p_1,p_2\in \Delta$ by the above discussion. So $g_k(p_1 - p_2)$ cannot be torsion and we are done.
\end{proof}

\begin{rmk}\label{RMK001}
Note that both \coref{COR800} and \ref{COR801} hold for very general curves or points.
Being very general means that they hold outside of a countable union of proper subvarieties. This is a notion only valid over an
uncountable field. If we work over number fields, e.g.,
$X$ and $Y$ are curves over a number field in Corollary \ref{COR801}, we may apply the generalized Bogomolov conjecture proved by S. W. Zhang
\cite{Z} to conclude that $g_k(\alpha(X))$ contains only finitely many torsion points over $\overline{\BQ}$. This implies that
the exceptional set $\Delta$ is a finite set so the Corollary holds over $\overline{\BQ}$ for $h(q)$ a general point on $Y$.
However, the application of Bogomolov conjecture seems an overkill for our purpose.
\end{rmk}

\section{A simple strange surface.}

Let $f: X \to Y$ be a finite morphism of degree $d\geq 2$ between two smooth projective curves with $g(Y) \geq 2$.
We may assume that $Y$ is a curve whose genus is maximum in the set of smooth projective curves of genus $\geq 2$
admitting a morphism $f_i: X \to Y_i$, i.e., satisfying the hypothesis of \coref{COR801}.
This can be justified by Theorem of de Franchis \cite{m} (which asserts that there are
finitely many pairs $(T, X \to T)$ of smooth projective curves $T$ of genus $\geq 2$ and a morphism $X \to T$). 

Now we choose a point $q \in Y$ such that $f^{-1}(q)$ contains at least two
distinct points $p_1\ne p_2 \in X$. We construct a curve $Z$ by gluing $p_1$ to $p_2$. More formally, it is 
the pushout:
$$
\xymatrix{
p_1 \coprod p_1 \ar[d]^f \ar[r] & X \ar[d]^g \\
q \ar[r]_j & Z
}
$$

Set $p= j(q) = g(p_1) =g(p_2) \in Z$. We observe the following:

\begin{itemize}
\item $Z$ is an algebraic curve with a node at $p$.

\item $H^1(Z,\Q)$ carries a mixed Hodge structure of weights $0$ and $1$. In particular, 
$$
\begin{aligned}
& H^1(Z,\Q) \cong H^1_c(X-f^{-1}(q),\Q), \quad   \gr^W_1 H^1(Z,\Q) \stackrel{g^*}{\cong} H^1(X,\Q).
\end{aligned}
$$
Furthermore, since $H^1_{\rm sm}(Z,\Q) $ is a sub-Hodge structure of $ \gr^W_1 H^1(Z,\Q)$, we have 
$$
\dim H^1_{\rm sm}(Z,\Q)\leq  \dim \gr^W_1 H^1(Z,\Q) = \dim H^1(X,\Q)=  2 g(X).
$$

\item There is a well-defined morphism $h: Z \to Y$ defined by $h(z) = f(g^{-1}(z))$ for any $z\in Z$, and hence the morphism $f:X \to Y$ factors through $Z$. 
\end{itemize}

Indeed, for $q\in Y$ very general, it follows immediately from \coref{COR801} that

\begin{prop}\label{PROP810}
Let $X,Y$ and $Z$ be the curves given above and let $q$ be a very general point on $Y$. Then
\begin{enumerate}
\item[(i)] $H^1_{\rm sm}(Z,\Q)  = h^* H^1(Y,\Q)$, and
\medskip

\item[(ii)] $H^1_{\rm sm}(Z \times Z ,\Q) = \bar{h}^* H^1(Y \times Y,\Q)$,
\end{enumerate}
where $h$ and $\bar{h} $ are the morphisms  in the following diagram:
$$
\xymatrix{
X \ar[rr]^f \ar@{>>}[dr]_g && Y && X \times X \ar[rr]^{\bar{f}} \ar[dr]_{\bar{g}}&& Y \times Y \\
& Z \ar[ur]_h& && &Z \times Z \ar[ur]_{\bar{h}}& 
}
$$
\end{prop}

More importantly, we claim that
\begin{equation}\label{E000816}
H^2_{\rm sm}(Z \times Z ,\Q) \subsetneq \gr^W_2 H^2(Z\times Z, \Q) \stackrel{\bar{g}^*}{\cong} H^2(X\times X, \Q).
\end{equation}
That is, $H^2(Z \times Z,\Q)$ carries a weight $2$ Hodge cycle not in $H^2_{\rm sm}(Z\times Z,\Q)$. Therefore $Z\times Z$ is strange.

Indeed, we can prove the following

\begin{prop}\label{PROP820}
Let $X$ be a smooth projective curve and let $g: X\to Z$ be the map gluing two distinct points $p_1$ and $p_2$ of $X$ as above.
If $p_1 - p_2$ is non-torsion in $\Pic(X)$, then $\bar{g}^* \Pic_\Q(Z\times Z) \ne \Pic_\Q(X\times X)$. In particular,
\begin{equation}\label{E000817}
\Delta_X + \pi_1^* D_1 + \pi_2^*D_2 \not\in \bar{g}^* \Pic_\Q(Z\times Z) 
\end{equation}
for all $D_1, D_2\in \Pic(X)$, where $\Delta_X$ is
the diagonal of $X\times X$ and $\pi_1$ and $\pi_2$ are the projections of $X\times X$ to $X$. Consequently,
\begin{equation}\label{E000818}
\Delta_X \not \in \bar{g}^* H_{\rm sm}^{1,1}(Z\times Z, \Q) \subsetneq H^{1,1}(X\times X, \BQ)
\end{equation}
and hence
\begin{equation}\label{E000819}
\bar{g}^* H_{\rm sm}^2(Z\times Z, \BQ) \subsetneq H^2(X\times X, \Q),
\end{equation}
where $H_{\rm sm}^{k,k}(W, \BQ) = H_{\rm sm}^{2k}(W, \BQ)\cap F^k H_{\rm sm}^{2k}(W, \BC)$ and ${\rm Pic}_{\Q}(W) = {\rm Pic}(W) \otimes \Q$.
\end{prop}

\begin{proof}
Note that $\bar{g}: X\times X\to Z\times Z$ factors through $Z\times X$ with the map $\varphi: X\times X\to Z\times X$. 
For a line bundle $L$ on $Z\times X$, we have
\[
\varphi^* L\Big|_{p_1\times X} = \varphi^*(L\Big|_{q\times X}) \text{ and }
\varphi^* L\Big|_{p_2\times X} = \varphi^*(L\Big|_{q\times X})
\]
where $q = g(p_1) = g(p_2)$. Thus,
\[
\varphi^* L\Big|_{p_1\times X} \cong j^* \varphi^* L\Big|_{p_2\times X}
\]
through the identification $j: p_1\times X\to p_2\times X$ sending $(p_1, x)\to (p_2,x)$. Therefore, if
$\Delta_X + \pi_1^* D_1 + \pi_2^* D_2$ lies in
$\varphi^* \Pic_\Q(Z\times X)$, then
\[
\mO_{p_1\times X}(N\Delta_X) \cong j^*\mO_{p_2\times X}(N\Delta_X)
\]
for some $N\in \BZ^+$. That is, $N(p_1 - p_2) = 0$ in $\Pic(X)$. This proves \eqref{E000817}.

Observe that
\begin{equation}\label{E000820}
\Pic(X\times X) = \pi_1^* \Pic^0(X) \oplus \pi_2^* \Pic^0(X) \oplus H^{1,1}(X\times X, \BZ)
\end{equation}
and
\begin{equation}\label{E000821}
\pi_1^* \Pic(X) \oplus \pi_2^* \Pic(X) \subset \bar{g}^* \Pic(Z\times Z).
\end{equation}
Combining \eqref{E000817}, \eqref{E000820} and \eqref{E000821}, we conclude \eqref{E000818}.

Finally, we assume that \eqref{E000819} fails to hold. That is, 
there exists a morphism $\eta: Z \times Z \to S$ from $Z\times Z$ to a smooth projective variety $S$ such that
\begin{equation}\label{E000822}
\rho^* H^2(S, \Q) = H^2(X\times X, \Q),
\end{equation}
where $\rho = \eta\circ \bar{g}$. By \eqref{E000818}, 
 there exists $\xi\ne 0 \in H^{1,1}(X\times X, \BQ)$ such that
$\xi$ is perpendicular to $\bar{g}^*H^{1,1}_{\rm sm}(Z \times Z,\Q)$.  
Since $\rho$ factors through $\bar{g}$, we have
$\xi \cdot \rho^* \omega = 0$ and hence $\rho_* \xi \cdot \omega = 0$ for all $\omega\in H^{1,1}(S, \BQ)$. Note that
$\rho_* \xi \in H^{n-1,n-1}(S, \BQ)$ for $n = \dim S$. By the Hard Lefschetz theorem, $H^{1,1}(S,\BQ)$ and $H^{n-1,n-1}(S, \BQ)$ are dual
to each other. Therefore, $\rho_* \xi = 0$ and it follows that $\xi \cdot \rho^* \omega = 0$ for all $\omega \in H^2(S,\Q)$.
By \eqref{E000822}, this implies $\xi = 0$. Contradiction.
\end{proof}

\begin{rmk}\label{RMK002}
Note that \propref{PROP820} holds for any pair $(X, Z)$, where $X$ is the normalization of a curve $Z$ with one node, as long as
$p_1 - p_2$ is non-torsion for the two points $p_1$ and $p_2$ over the node. In our setting, we are expected to say more. Indeed, we believe
that $H^2_{\rm sm}(Z \times Z ,\Q) = \bar{h}^* H^2(Y \times Y,\Q)$ and hence
\begin{equation}\label{E000823}
H^k_{\rm sm}(Z \times Z ,\Q) = \bar{h}^* H^k(Y \times Y,\Q)
\end{equation}
for all $k$. But we do not know how to prove \eqref{E000823} yet.
\end{rmk}

\section{Revisiting the example of Barbieri-Viale and Srinivas} 

 The previous example had singularities in codimension one. An example
 of a strange normal surface was constructed in \cite{bs}. We recall
 the relevant details.
Let $X$ be a hypersurface in $\PP^3$ defined by an equation
$F(x,y,z,w) = w(x^3 - y^2 z) + f(x,y,z) $ where $f(x,y,z)$ is a
general homogeneous polynomial of degree 4 in $x,y$ and $z$. $X$ has an
isolated singularity at $p=[0:0:0:1]$.  So it is normal. In \cite{bs}, the authors showed the following:

\medskip\noindent
(a) Let $f: Y \to X$ be the blow-up of $X$ at $p$ and $E$ be the exceptional divisor. Then $\pi: Y \to \PP^2 $ is also a blow-up of $\PP^2$ at 12 points $\{p_1, \cdots , p_{12}\} = V(x^3 -y^2 z) \cap V(f(x,y,z))$, and hence $Y$ is a smooth rational surface.

\medskip \noindent
(b) There is an exact sequence
$$
0 \to H^2(X,\Q) \to H^2 (Y, \Q) \stackrel{\alpha}{\to} \Q \to 0 
$$
where $\alpha$ is the intersection number with the cohomology class of
$E$.  It follows that $H^2(X,\Q)$ is  $12$ dimensional because
$H^2(Y,\Q)$ is generated by the pullback of the class $h$ of a line in
$\PP^2$ and the $12$ exceptional divisor classes $\{ e_1, \cdots,
e_{12}\}$. Note also that $H^2(Y,\Q)$ and therefore $H^2(X,\Q)$ consists entirely of Hodge classes.

\medskip \noindent
 (c) The classes $e_i-e_j \notin \im [c_1: {\rm Pic}(X) \otimes \Q \to H^2(X,\Q)]$.
 \medskip
 
It follows from Corollary \ref{cor:c1toH2sm} that $e_i-e_j\notin H^2_{\rm sm}(X,\Q)$. So $X$ is definitely strange.
More precisely, we have $\Pic(X)\cong \BZ$ by the following lemma and hence $H^2_{\rm sm}(X, \BQ) \cong \BQ$.

\begin{lemma}\label{LEM002}
For $f(x,y,z)$ general, $\Pic(X)$ is freely generated by $\mO_X(1)$ over $\BZ$.
\end{lemma}

\begin{proof}
The blowup $\pi: Y\to \PP^2$ is actually the composition of $g: X\dashrightarrow \PP^2$ and
$f: Y\to X$, where $g$ is the projection sending $[x:y:z:w]$ to $[x:y:z]$. Clearly, $g$ is regular outside
of $p$ and blowing up $X$ at $p$ resolves the indeterminacy of $g$; the resulting regular map $Y\to \PP^2$
is exactly $\pi$. Alternatively, we can construct $Y$ and $X$ from $\PP^2$ as follows.

Let $C$ be the cuspidal cubic curve given by $x^3 - y^2 z = 0$ on $\PP^2$. It is well known that
$\Pic(C) = \BZ \oplus \mathbb{G}_a$, where $\mathbb{G}_a$ is the additive group of $\BC$. Obviously, we have an injection
$\Pic(\PP^2)\hookrightarrow \Pic(C)$. We choose $12$ points $p_1, p_2, ..., p_{12}$ on 
$C\backslash \{[0:0:1]\}$ such that
\begin{itemize}
\item $4h = p_1 + p_2 + ... + p_{12}$ in $\Pic(C)$ and
\item $p_1, p_2, ..., p_{12}$ are linearly independent over $\BQ$ in $\Pic(C)\otimes \BQ$.
\end{itemize}
Here we use $h$ for both the hyperplane class in $\PP^2$ and, for convenience, its pullback to $C$.

By the surjection $H^0(\mO_{\PP^2}(4))\to H^0(\mO_C(4))$, we see that there exists a quartic curve
$D = V(f(x,y,z))$ passing through $p_1, p_2, ..., p_{12}$. Let $\pi: Y\to \PP^2$ be the blowup of $\PP^2$
at $p_1, p_2, ..., p_{12}$ and let $E$ and $F\subset Y$ be the proper transforms of $C$ and $D$, respectively.
Note that
\[
E = 3\pi^* h - e_1 - e_2 - ... - e_{12} \text{ and }
F = 4\pi^* h - e_1 - e_2 - ... - e_{12}.
\]
By the exact sequence
\[
\xymatrix{
0 \ar[r] & H^0(\mO_Y(F - E)) \ar[r] \ar@{=}[d]& H^0(\mO_Y(F)) \ar[r] & H^0(\mO_E(F)) \ar[r]\ar@{=}[d] & 0\\
         & H^0(\mO_Y(\pi^* h))      & & H^0(\mO_E)}
\]
we see that $|F|$ is a base point free linear series of dimension $3$. Let $f: Y\to \PP^3$ be the map given by $|F|$.
Since $E\cdot F = 0$, $f_* E = 0$, i.e., this map contracts the curve $E$ to a point.
It is exactly the map that maps $Y$ onto $X$ at the very beginning
of this section.

The Leray spectral sequence for the sheaf $\mO^*_Y$ gives an exact sequence 
$$
0 \to \Pic (X) \to \Pic (Y) \to H^0(X, R^1 f_*\mO^*_Y).
$$
By composing the last map with a restriction $H^0(X, R^1 f_*\mO^*_Y) \to \Pic(E)$, we can see that $\Pic(X)$ lies in the kernel of the map $\Pic(Y) \to \Pic(E)$.
On the other hand, for every
divisor $M = d\pi^* h + m_1 e_1 + m_2 e_2 + ... + m_{12} e_{12}$ in $\Pic(Y)$,
$$
(d \pi^* h + m_1 e_1 + m_2 e_2 + ... + m_{12} e_{12})\bigg|_E = dh + m_1 p_1 + m_2 p_2 + ... + m_{12} p_{12}
$$
in $\Pic(E)\cong \Pic(C)$. Since we choose $p_1,p_2, ...,p_{12}$ to be linearly independent over $\BQ$,
$dh + m_1 p_1 + m_2 p_2 + ... + m_{12} p_{12} = 0$ in $\Pic(C)$ if and only if
$$
m_1 = m_2 = ... = m_{12} = -\frac{d}4.
$$
That is, $M$ lies in the kernel of $\Pic(Y)\to \Pic(E)$ if and only if $M$ is a multiple of $F$. It
follows that $\Pic(X)$ is generated by $F = \mO_X(1)$.
\end{proof}

\section{Examples of varieties with normal crossings}
In this section we consider two examples of varieties with normal crossings, one constructed by Bloch \cite{Jannsen} and
the other constructed by Srinivas \cite{bv}. Each of these examples will be reviewed after we establish a couple of lemmas that we will need later.

\medskip

Let $Y$ be a smooth projective variety and $Z$ be a smooth subvariety of $Y$. Let $X = Y \coprod_Z Y$ be the variety obtaining by glueing two copies of $Y$ along $Z$. i.e., $X$ is the variety defined as the pushout
$$
\xymatrix{
Z \coprod Z \ar[r]^{i \coprod i \quad } \ar[d] & Y \coprod Y \stackrel{\rm set}{=} \tilde{Y} \ar[d]^{\pi} \\
Z \ar[r]^j & X
}
$$
where $i: Z \hookrightarrow Y$ and $j: Z \hookrightarrow X$ are inclusions. Furthermore, $\pi: \tilde{Y}= Y \coprod Y \to X$ is the desingularization of $X$. This is the disjoint union of the
inclusions $i_j : Y \hookrightarrow X$ (for $j=1,2$)  of the two components.
Since these are regular embeddings, there are pullbacks $i_j^*$, and therefore $\pi^*$, on the level of Chow groups \cite[chap 6]{fulton1}. 

\begin{lemma}\label{lemma:cl2-vs-c2}
There exists a commutative diagram whose top row is a complex and bottom row is exact: 
\begin{equation}\label{eq:diag-1-NC}
\xymatrix@=18pt{
&\Ch^p(X;\Q) \ar[r]^{\pi^*} \ar[d]^{\cl^X_p}& \Ch^p(\tilde{Y} ;\Q) \ar[d]^{\cl^{\tilde{Y}}_p} \ar[r]^{\iota^*} & \Ch^p(Z;\Q) \ar[d]^{\cl^Z_p} \\
0 \ar[r] & \gr^W_{2p}H^{2p}(X,\Q) \ar[r]^{\pi^* \qquad} & H^{2p}(Y,\Q) \oplus H^{2p}(Y,\Q) \ar[r]^{\quad \qquad \iota^*} & H^{2p}(Z,\Q)  
}
\end{equation}
where the last two maps labelled $\cl^*_p$ are the $p$-th cycle class maps, and $\iota^*$ is the difference of the restrictions.
Furthermore, 
$$
\im [\cl^X_p : \Ch^p(X;\Q) \to H^{2p}(X,\Q)] \supseteq \im [c_p : K^0(X) \otimes \Q \to H^{2p}(X,\Q)].
$$
\end{lemma}

\begin{proof}
For the exactness of the bottom row, we use the Mayer-Vietoris sequence
$$
\cdots \to H^{2p-1}(Z,\Q)   \to H^{2p}(X,\Q) \to H^{2p}(Y,\Q) \oplus H^{2p}(Y,\Q) \to H^{2p}(Z,\Q) \to \cdots 
$$
and take the exact functor $\gr^W_{2p}$.

As for the existence of $\cl^X_p$, since $\iota^*(\cl^{\tilde{Y}}_p(\pi^*(\xi))) = \cl^Z_p (\iota^*\pi^*(\xi)) =0$ for any $\xi \in \Ch^p(X;\Q)$, we have $\cl^{\tilde{Y}}_p(\pi^*(\xi))  \in \ker (\iota^*) =  \im(\pi^*)$. Injectivity of $\pi^*$ gives rise to a unique $\alpha \in H^{2p}(X,\Q)$ such that $\pi^*(\alpha ) = \cl^{\tilde{Y}}_p (\pi^*(\xi))$. Define $\cl^X_p(\xi) = \alpha$. Commutativity of the diagram \eqref{eq:diag-1-NC} follows from the definition. 

For the last statement, it is enough to show that $c_p(\scE)\in H^{2p}(X,\Q)$ lifts to $\Ch^p(X;\Q)$ for a vector bundle $\scE$ on $X$. Fulton \cite[chap 3]{fulton1} defines Chern classes, which we denote by $c_p^{\rm fulton}$, as operators $c_p^{\rm fulton}(\scE): \Ch^*(X;\Q)\to \Ch^{*+p}(X;\Q)$. Letting $\Xi =c_p^{\rm fulton}(\scE)([X])\in \Ch^p(X;\Q)$, and using the compatibilities given \cite[chap 19]{fulton1} shows that $\Xi$ maps to $c_p(\scE)$ under the cycle map.
 \end{proof}

By the universal property of the pushout, there exists a unique map $q: X \to Y$ such that the following diagram commutes. 
\begin{equation}\label{eq:diag-2-NC}
\xymatrix{
Z \coprod Z \ar[r]^{i \coprod i \quad } \ar[d] & Y \coprod Y \stackrel{\rm set}{=} \tilde{Y} \ar[d]^{\pi} \ar@/^1pc/[ddr] &  \\
Z \ar[r]^j  \ar@/_1pc/[drr] & X \ar[dr]^q & \\
&& Y
}
\end{equation}

\begin{lemma}\label{lemma:nc-3}
In the notation as above, we have
\begin{equation}\label{eq:nc-sm}
q^*(H^{2p}(Y,\Q)) \cong \{(\alpha,\alpha) \mid \alpha \in H^{2p}(Y,\Q) \} \subseteq  H^{2p}_{\rm sm}(X,\Q).
\end{equation}
Furthermore, if the Hodge conjecture holds for $Y$ in degree $2p$, then
$$
q^*H^{p,p}(Y,\Q) \subseteq \im (\ch^X_p) 
$$
where $\ch_p$ denotes the $p$-th component of the Chern character.
\end{lemma}
\begin{proof}
\eqref{eq:nc-sm} follows immediately from the definition of the smooth center. For the second statement,
assume that the Hodge conjecture holds for $Y$ in degree $2p$.
Then any Hodge $(p,p)$-cycle $\alpha$ can be written as $\alpha = \sum_{i} k_i \, \ch^Y_p ([\scE_i])$
where $\scE_i$ are vector bundles on $Y$, and $k_i \in \Q$. Since $\pi^*$ is injective and
$$
\pi^*(q^*(\alpha))= \pi^* q^*\left(\sum_i k_i \,\ch^Y_p([\scE_i]) \right) = \pi^* \ch^X_p  \left(\sum_i k_i\, [q^*(\scE_i)]\right),
$$
it follows that
$$
q^*(\alpha) = \ch^X_p  \left(\sum k_i\, [q^*(\scE_i)]\right) \in \im(\ch^X_p).
$$
Therefore, any $\alpha \in H^{p,p}(Y,\Q)$ gives rise to a class $q^*(\alpha) \cong (\alpha, \alpha)$ in $ \im(\ch^X_p) $. 
\end{proof}

\subsection{Bloch's example} (cf. \cite[Appendix A]{Jannsen, lewis})
Let $S_0$ be a general hypersurface in $\PP^3$ defined over $\ol{\Q}$ of degree $d \geq 4$, and $p \in S_0(\C)$ be a $\ol{\Q}$-generic point. Let $\sigma: P = {\rm Bl}_p \PP^3 \to \PP^3$ and $S = {\rm Bl}_p S_0$, and set $X = P \coprod_S P$. The Mayer-Vietoris sequence gives rise to an exact sequence
\begin{equation}\label{eq:bloch-1}
0 \to H^4(X,\Q) \to H^4(P,\Q) \oplus H^4(P,\Q) \to H^4 (S,\Q) \to 0.
\end{equation}
Let $h$ be the cohomology class of a general hyperplane in $\PP^3$ and $e = [E]$ be the cohomology class of the exceptional divisor of the blow-up $\sigma: P \to \PP^3$. Then
$$
H^4(S,\Q)\cong \Q \quad \text{and} \quad H^4(P,\Q) \cong H^4 (\PP^3, \Q) \oplus H^0(p_0,\Q) = \Q \, h^2 \oplus \Q \, e^2
$$ 
with intersection numbers $(h^2, S)_P = d$ and $(e^2 , S)_P = -1$.
From \eqref{eq:bloch-1}, we get 
\begin{equation}\label{eq:X-Bloch}
\begin{aligned}
H^4(X,\Q) &\cong \{(\alpha, \beta) \in H^4(P,\Q) \oplus H^4(P,\Q) \mid (\alpha, S)_P - (\beta,S)_P =0 \}\\
&=\Q \, (h^2, -de^2) \oplus  \Q \, (e^2, e^2)  \oplus \Q \, (0, h^2+de^2).
\end{aligned}
\end{equation}
This consists entirely of Hodge (2,2)-cycles.

\medskip
In the letter to Jannsen  \cite{Jannsen}, Bloch showed that a cycle $((d-1)h^2, \, d(h^2+e^2))$ in $H^4(X,\Q)$ cannot be in the image of $\cl^X_2$, and hence 
 \begin{equation}\label{eq:B-cycle}
\gamma \stackrel{\rm set}{=} ((d-1)h^2, \, d(h^2+e^2)) \notin  \im [\ch_2: K^0(X)\otimes \Q \to H^4(X,\Q)].
 \end{equation}
by Lemma \ref{lemma:cl2-vs-c2}. Let 
\begin{equation}\label{eq:frak-H}
\mathfrak{H} = \sum \{h^*(\beta) \in H^4_{\rm sm}(X,\Q) \mid (T, \, h: X \to T)\in \mathcal{C}(X),\, \beta \in H^4(T,\Q)_{\rm alg} \} 
\end{equation}
be the subspace of cycles in $H^4_{\rm sm}(X,\Q)$ coming from algebraic cycles in degree 4 on a smooth projective variety.

\begin{prop}\label{ex:Bloch}
Let $X$ be the variety  above. Then, 
\begin{enumerate}
\item [(i)] $\im (\ch_2)  \cong \Q \,  (e^2, e^2)  \oplus  \Q \,  (h^2, h^2)$,
\item [(ii)] $\mathfrak{H}  \subseteq q^*(H^4(P,\Q)) $, 
\item [(iii)] If $H^4_{\rm sm}(X,\Q)=\mathfrak{H}$ (e.g. if the Hodge conjecture holds), then $H^4_{\rm sm}(X,\Q) = \im (\ch_2)  = q^*(H^4(P,\Q)) \subsetneq H^4(X,\Q)$. 
\end{enumerate}
where $q: X\to P$ is the map as in the diagram \eqref{eq:diag-2-NC}.
\end{prop}

\begin{proof}
Since the Hodge conjecture holds for $P$ in degree 4, by Lemma \ref{lemma:nc-3} we have $q^*(H^4(P,\Q)) \subseteq \im (\ch_2)$. Then
$$
2 = \dim q^*(H^4(P,\Q))  \leq \dim \, \im(\ch_2)<\dim H^4(X,\Q) =3.
$$
Hence 
$\im (\ch_2) = q^*(H^4(P,\Q)) \cong \Q \, (h^2,h^2) \oplus \Q \, (e^2,e^2)$ as we claimed in (i).

\medskip
In order to show (ii), let $\alpha \in \mathfrak{H} \subseteq H^4_{\rm sm} (X,\Q)$ and let $t: X \to T$ be a morphism to a smooth projective variety $T$ such that 
$\alpha = t^*(\beta ) $ and $\beta \in H^4(T,\Q)_{\rm alg}$.  
\begin{equation}\label{eq:diag-3-NC}
\xymatrix@=20pt{
P \ar@{^(->}[r]_{\iota_1} \ar@/^1pc/[rr]^{\rm id_P} & X \ar[r]_q \ar[d]^t & P&& H^4(P,\Q) \ar@/^1pc/[rr]^{{\rm id}^*_P}\ar[r]_{q^*} & H^4(X,\Q)\ar[r]_{\iota^*_1} & H^4(P,\Q) \\
&T& && &H^4(T,\Q) \ar[u]^{t^*}&
}
\end{equation}
 Let $\iota_j: P \to X$ be the composition of the canonical injection $P \to P \coprod P$ followed by $\pi: P \coprod P \to X$ for $j=1,2$. Since $q \circ \iota_j = {\rm id}_P$ for $j=1,2$, we have
$$
\iota^*_1 ( \alpha - q^*(\iota^*_1(\alpha) )) = \iota^*_1(\alpha) - \iota^*_1 q^*\iota^*_1 (\alpha) = \iota^*_1(\alpha) - \iota^*_1(\alpha)=0,
$$
and hence 
$$
\alpha - q^*(\iota^*_1(\alpha)) \in \ker[\iota^*_1: H^4(X,\Q) \to H^4(P,\Q)] \stackrel{\rm (*)}{\cong} \gr^W_4 H^4_c (P-S,\Q) \cong (h^2+ de^2) \cdot \Q
$$
where ${\rm (*)}$ is from the following commutative diagram with exact rows
\begin{equation}\label{eq:Bloch-kernel}
\xymatrix{
0 \ar[r] & \gr^W_4 H^4_c (X-P,\Q)\ar[d]^{\cong} \ar[r] & H^4(X,\Q) \ar[r]^{\iota^*_1}  \ar[d]^{\iota^*_2}& H^4(P,\Q) \ar[d]^{j^*}  & \\
0 \ar[r] & \gr^W_4 H^4_c(P-S,\Q) \ar[r] & H^4(P,\Q) \ar[r]^{j^*} & H^4(S,\Q) \ar[r] & 0  
}
\end{equation}
Hence, there exists $k \in \Q$ such that
\begin{equation}\label{eq:Bloch-k}
k\, (0, h^2+de^2) =\alpha - q^*(\iota^*_1(\alpha)) =t^*(\beta) - q^*(\iota^*_1(\alpha)). 
\end{equation}
We show that $k=0$. Suppose $k \neq 0$. Observe
$$
\begin{aligned}
\gamma= ((d-1) h^2 ,& \, d(h^2+e^2)) = (d-1) (h^2, h^2) + (0, h^2+de^2)  \\
&= (d-1) q^*(h^2) + \frac{1}{k}\left(t^*(\beta) -q^*(\iota^*_1(\alpha)) \right)  \\
&= q^*( (d-1) h^2 -  \frac{1}{k}\, \iota^*_1 (\alpha)) + \frac{1}{k}\, t^*(\beta) \in q^*H^4(P,\Q)  + t^* H^4(T,\Q)_{\rm alg} 
\end{aligned}
$$
i.e., $\gamma \in \mathfrak{H}$. Since any algebraic cycle in a smooth projective variety can be realized as a finite sum of Chern classes of vector bundles on the variety, this observation implies $\gamma \in \im (\ch_2 ) \otimes \Q$, which contradicts to \eqref{eq:B-cycle}. Thus $k=0$. Now \eqref{eq:Bloch-k} implies $\alpha  = q^*(\iota^*_1(\alpha)) \in q^*(H^4(P,\Q))$. Hence, $\mathfrak{H} \subseteq q^*(H^4(P,\Q))$. 

\medskip 
For (iii), suppose $H^4_{\rm sm}(X,\Q) = \mathfrak{H}$. Then, (ii) and Lemma \ref{lemma:nc-3} imply that
$$
q^*( H^4(P,\Q)) \subseteq \im (\ch_2) \otimes \Q \subseteq H^4_{\rm sm}(X,\Q) =\mathfrak{H}\subseteq q^*(H^4(P,\Q)).
$$
Therefore, $\im (\ch_2) \otimes \Q = H^4_{\rm sm}(X,\Q) =  q^*(H^4(P,\Q)) =\Q \,(e^2, e^2) \oplus \Q \, (h^2,h^2)$.
\end{proof}

\begin{rmk}
In fact, for this $X$, the de Rham filtration $F^2_{\rm DR}H^4(X,\C)$ \eqref{eq:dr-filtration} coincides with the Hodge filtration $F^2 H^4(X,\C)$. In order to show this, it is enough to prove $F^2_{\rm DR}H^4(X,\C) \supseteq F^2 H^4(X,\C)$. Let $\tau^1_X$ be the kernel of $\Omega^1_X \to \pi_* \Omega^1_{\tilde{P}}$ where $\pi: \tilde{P} = P \coprod P \to X$ is the desingularization of $X$,  and let $\tau^k_X = \wedge^k \tau^1_X$. There exists an exact sequence \cite[Proposition 1.5 (1)]{friedman}:
$$
0 \to \tau^k_X \to \Omega^k_X \to \pi_* \Omega^k_{\tilde{P}} \to i_* \Omega^k_S \to 0  \quad \text{for $k \geq 1$}
$$
where $i:  S \hookrightarrow X$. We can split this into two short exact sequences:
\begin{equation}\label{eq:ses-1}
0 \to \tau^k_X \to \Omega^k_X {\to} \bar{\Omega}^k_X \to 0,\qquad 0 \to \bar{\Omega}^k_X \to \pi_* \Omega^k_{\tilde{P}} \to i_* \Omega^k_S \to 0 
\end{equation}
where $\bar{\Omega}^k_X  = \Omega^k_X / \tau^k_X$. By direct computation in cohomologies associated to short exact sequences \eqref{eq:ses-1}, we get
$$
H^2(X, \bar{\Omega}^2_X) = \rho (H^2(X,\Omega^2_X)) \subseteq  \rho \left( \im[\bH^4(X,\Omega^{\geq 2}_X) \to \bH^4(X, \Omega^{\geq \dt}_X)]\right) = F^2_{\rm DR} H^4(X,\C).
$$
where $\rho: \Omega^{\dt}_X \to \bar{\Omega}^{\dt}_X$ is the morphism of filtered complexes \cite{dubois, ArapuraKang}. Since $F^2H^4(X,\C) = \bigoplus_{p+q=4} H^q (X, \bar{\Omega}^p_X)  = H^2(X,\bar{\Omega}^2_X)$ by (\cite[Proposition 1.5]{friedman}), we are done. 
 
Since $H^4(X,\Q) $ carries a pure Hodge structure of type $(2,2)$, this is an example that the smooth center $H^4_{\rm sm}(X,\Q)$ provides a stronger constraint than the de Rham filtration in degree 4 that we considered in \cite{ArapuraKang}.
\end{rmk}

\subsection{Srinivas' Example} \cite{bv} Let $Y$ be a smooth hypersurface of degree 2 in $\PP^5$ and $Z$ be a subvariety of $Y$ cut by a smooth hypersurface of degree $d \geq 3$. Then $\deg \, Z = 2d \geq 6$ and $H^{3,0}(Z) \neq 0$. Let $X = Y \coprod_Z Y$. By taking an exact functor $\gr^W_4$ on the Mayer-Vietoris exact sequence, we get
\begin{equation}\label{eq:srinivas-2}
0 \to \gr^W_4 H^4(X,\Q) \to H^4(Y,\Q)^{\oplus 2} \stackrel{s}{\to} H^4(Z,\Q) \to \gr^W_4 H^5(X,\Q) \to 0 
\end{equation}
We have the Lefschetz decomposition: 
\begin{equation}\label{eq:primitive-decomposition}
H^4(Y,\Q) \cong \bigoplus^2_{k=0} L^k H^{4-2k}(Y,\Q)_{\rm prim}
\end{equation}
where $H^i (Y,\Q)_{\rm prim} = \ker[L^{5-i}: H^i(Y,\Q) \to H^{10-i}(Y,\Q)]$ is the $i$-th primitive cohomology of $Y$. Note
$$
\begin{aligned}
H^4 (Y,\Q)_{\rm prim} &= \ker [L: H^4(Y,\Q) \to H^6(Y,\Q)], \\
H^2 (Y,\Q)_{\rm prim} &= \ker[L^3: H^2(Y,\Q) \to H^8(Y,\Q)] = \{0\}, \\
H^0(Y,\Q)_{\rm prim} &=\ker[L^5: H^0(Y,\Q) \to H^{10}(Y,\Q)=0] = H^0(Y,\Q),
\end{aligned}
$$
Furthermore, since $Y$ is a quadric hypersurface of even dimension, it contains two families of planes \cite[Proposition p.735]{gh}. Let $\alpha \in \Ch^2(Y;\Q)$ be the equivalence class of difference of two planes belonging to different families of planes. Then, its cohomology class $\bar{\alpha}\stackrel{\rm set}{=} \cl^Y_2(\alpha) \in H^4(Y,\Q)_{\rm prim}$ and $i^*(\bar{\alpha}) =0$; on the other hand, $H^{3,0}(Z) \neq 0$ implies that $i^*(\alpha)$ is a nonzero class in the Griffiths group $\rm{Griff}^2(Z;\Q)=\Ch^2(Z;\Q)_{\rm hom}/\Ch^2(Z;\Q)_{\rm alg}$ where $i: Z \hookrightarrow Y$. An explanation can be found in \cite[\S 5.2]{bv}. Since $\dim H^4(Y;\Q)_{\rm prim} =1$, we may choose $\bar{\alpha}$ as a generator of $H^4(Y,\Q)_{\rm prim}$. By putting these observation together, we get
$$
H^4(Y,\Q) = H^4(Y,\Q)_{\rm prim} \oplus \Q \, h^2 = \Q \, \bar{\alpha}  \oplus \Q\, h^2 .
$$
where $h = [Y \cap H] \in H^2(Y,\Q)$ is a cohomology class of a general hyperplane section of $Y$. Then from \eqref{eq:srinivas-2} we have
$$
\gr^W_4 H^4(X,\Q)  = \ker[s: H^4(Y,\Q)^{\oplus 2} \to H^4(Z,\Q)] \cong  \Q\, (\bar{\alpha},0) \oplus \Q\, (0, \bar{\alpha}) \oplus \Q \, (h^2,h^2). 
$$
Diagram \eqref{eq:diag-1-NC} implies that $(\bar{\alpha},0)$ and $(0, \bar{\alpha})$ are nonzero Hodge $(2,2)$-cycles in $\gr^W_4 H^4(X,\Q)$ not contained in $\im [ \cl^X_2: \Ch^2(X;\Q)\to H^4(X,\Q)]$. Thus by Lemma \ref{lemma:cl2-vs-c2}, we have $(\bar{\alpha},\,0) ,\, (0, \, \bar{\alpha}) \notin \im[\ch_2: K^0(X) \otimes \Q \to H^4(X,\Q)]$. 

\begin{prop}
Let $X$ be the variety constructed by Srinivas \cite[5.2]{bv}. Then 

\begin{enumerate}
\item [(i)] $\im(\ch_2) \otimes \Q = \Q\,  (\bar{\alpha},\, \bar{\alpha}) \oplus \Q\,  (h^2,\, h^2) $

\item [(ii)] $\mathfrak{H} \subseteq q^*(H^4(Y,\Q))$, if the Hodge conjecture holds for $Y$ in degree 4, where $\mathfrak{H}$ is defined in \eqref{eq:frak-H}, 
 
 \item [(iii)] If  $H^4_{\rm sm}(X,\Q) =\mathfrak{H}$, then 
 $H^4_{\rm sm}(X,\Q) = \im (\ch_2) \otimes \Q = q^*(H^4(Y,\Q)) \subsetneq \gr^W_4 H^4(X,\Q)$.
 \end{enumerate}
\end{prop}

\begin{proof}
Exactly same argument as in Bloch's example shows (i). For (ii) and (iii), let $\alpha \in  \mathfrak{H}$. By diagrams \eqref{eq:diag-3-NC} and \eqref{eq:Bloch-kernel} (replace $P$ and $S$ by $Y$ and $Z$, respectively) we have
$$
\alpha - q^*(\iota^*_1(\alpha) ) \in \ker[\iota^*_1: H^4(X,\Q) \to H^4(Y,\Q)] \cong \gr^W_4 H^4_c( Y-Z,\Q)  = \bar{\alpha}\cdot \Q.
$$
Hence, $k \,(0,\bar{\alpha}) = \alpha - q^*(\iota^*_1(\alpha))$ for some $k \in \Q$. Since we assume the Hodge conjecture for $Y$ in degree 4,  we can use the same argument as in Example \ref{ex:Bloch} to conclude $k=0$ and we get (ii). The proof of (iii) is exactly same as the one given in the Example \ref{ex:Bloch}. 
\end{proof}

\section{More on varieties with normal crossings}

We can construct more varieties with normal crossings $X = Y_1\coprod_Z Y_2$ similar to Bloch's example
\ref{ex:Bloch} with the property
\begin{equation}\label{E000795}
H^4(X, \BQ) \supsetneq H_{\text{sm}}^4(X, \BQ) = \im(c_2)\otimes \BQ
\end{equation}
and we can prove \eqref{E000795} without assuming the Hodge conjecture as in \propref{ex:Bloch}.

Let $X = Y_1\coprod_Z Y_2$ be the variety obtained by glueing two smooth projective varieties $Y_1$
and $Y_2$ transversely along a smooth hypersurface $Z$ in both $Y_1$ and $Y_2$. The correct way to think
of it is that $X$ is given by two embeddings $i_k: Z\hookrightarrow Y_k$ for $k=1,2$ with the Picard group
of $X$ given by
\begin{equation}\label{E000796}
0\xrightarrow{} \Pic(X) \xrightarrow{} \Pic(Y_1)\oplus \Pic(Y_2)
\xrightarrow{i_1^* L_1 - i_2^* L_2} \Pic(Z).
\end{equation}
Such $X$ is not necessarily projective. It is projective if and only if there are ample line bundles
$L_1$ and $L_2$ on $Y_1$ and $Y_2$, respectively, such that $i_1^* L_1 = i_2^* L_2$ on $Z$. Namely,
$L = L_1\oplus L_2$ gives an ample line bundle on $X$ and we can embed $X$ to $\PP^N$ by $|mL|$.

\begin{prop}\label{PROP802}
Let $X = Y_1\coprod_Z Y_2$ be a projective 3-fold with normal crossings satisfying
\begin{equation}\label{E000800}
H^4(Y_k,\Q) = H^{2,2}(Y_k,\Q) \text{ for } k = 1,2 \text{ and }
H^3(Z,\Q) = 0.
\end{equation}
Then $H_{\rm sm}^4(X, \BQ)$ is algebraic in the sense that
\begin{equation}\label{E000802}
H_{\rm sm}^4(X, \BQ) \subset \cl_2(V)
\end{equation}
where $V$ is the kernel of the map $\Ch^2(Y_1; \BQ) \oplus \Ch^2(Y_2; \BQ)
\to
\Ch^2(Z; \BQ)$ sending $(\xi_1, \xi_2)$ to $i_1^* \xi_1 - i_2^* \xi_2$
with $i_1$ and $i_2$ being the embeddings $i_1: Z\hookrightarrow Y_1$ and $i_2: Z\hookrightarrow Y_2$,
respectively.
\end{prop}

We use an argument akin to Lefschetz pencil. Basically, for every map $f: X\to W$ from $X$ to a smooth
projective variety $W$, we can ``fiberize'' $W$ to a family $W/B$
of 3-folds with $f(X)$ contained in a fiber.
Using the fact that the Hodge conjecture holds for 3-folds, i.e.,
for the fibers of $W/B$, we can show that the pull back $f^*\omega$ is
algebraic for every $\omega\in H^4(W, \BQ)$. As in the case of the classical
Lefschetz pencil argument, we need to know the
type of the singularities that a fiber $W_b$ of $W/B$ has. For that purpose, we first
prove the following lemmas:

\begin{lemma}\label{LEM000}
 Let $W$ be a smooth projective variety and $L$ be a line bundle on $W$.
Suppose that $\mO_W(-2p_1 - 2p_2 - ... - 2p_m)$ imposes independent conditions
on $H^0(L)$ for all $m$-tuples of distinct points $p_1, p_2, ..., p_m$ of $W$, i.e.,
the map
\begin{equation}\label{E000805}
 H^0(L)\twoheadrightarrow H^0(L\otimes \mO_W/\mO_W(-2p_1 -2p_2 - ... -2p_m))
\end{equation}
is surjective, where $\mO_W(-lp) = I_p^l$ with $I_p$ the ideal sheaf of a point $p\in W$. 
Then for a general linear subspace $B$ of
$|L| = \PP H^0(L)$ of $\dim B < m$, every member
$S\in B$ has at worst $m-1$ isolated singularties.
\end{lemma}

\begin{proof}
Let $U$ be the open set of $W^m$ consisting of $m$-tuples $(p_1, p_2, ...,
p_m)$ of distinct points of $W$ and $V\subset U\times |L|$ be the incidence correspondence
consisting of $(p_1, p_2, ..., p_m, S)$ satisfying $S\in \PP H^0(L\otimes \mO_W(-2p_1 - 2p_2
- ... - 2p_m))$.

Let $\pi_1$ and $\pi_2$ be the projection $V\to U$ and $V\to |L|$, respectively. Note that
\eqref{E000805} is equivalent to saying that
\begin{equation}\label{E000806}
 h^0(L\otimes \mO_W(-2p_1 - 2p_2 - ... - 2p_m)) = h^0(L) - m(n+1)
\end{equation}
where $n = \dim W$. Therefore, every fiber of $\pi_1$ has dimension $\dim |L| - m(n+1)$ and
hence $\dim V = \dim |L| - m$. It follows that $\pi_2(V)$ has dimension at most
$\dim |L| - m$.
So a general linear subspace $B$ of dimension $\dim B < m$ is disjoint from $\pi_2(V)$.

Clearly, a member $S\in |L|$ that has $\ge m$ isolated singularities or has singularities
along $\Gamma\subset S$ of $\dim\Gamma > 0$ belongs to
$\PP H^0(L\otimes \mO_W(-2p_1 - 2p_2
- ... - 2p_m))$ for some $(p_1, p_2, ..., p_m)\in U$. That is, such $S$ lies in $\pi_2(V)$.
Therefore, every $S\in B$ has at worst $m-1$ isolated singularities.
\end{proof}

\begin{lemma}\label{LEM001}
With the same hypotheses of \lemref{LEM000}, we further assume that $\mO_W(-3p)$ imposes
independent conditions on $H^0(L)$, i.e., the map
\begin{equation}\label{E000807}
H^0(L)\twoheadrightarrow H^0(L \otimes \mO_W/\mO_W(-3p))
\end{equation}
is surjective for all $p\in W$. For a general linear subspace $B$ of $|L|$,
\begin{enumerate}
 \item every member $S\in B$ has at worst $m-1$ isolated double points
if
\begin{equation}\label{E000808}
\dim B < \min\left(m, \binom{n+2}{2} - n\right)
\end{equation}
where $n = \dim W$;
\item every member $S\in B$ has at worst $m-1$ isolated double points of rank $\ge n-1$
if
\begin{equation}\label{E000809}
 \dim B < \min\left(m, 4\right)
\end{equation}
where we say $S$ has an isolated double point at $p$ of rank $\ge n-1$
if it is locally cut out on $W$ by
\begin{equation}\label{E000814}
x_1^2 + x_2^2 + ... + x_{n-1}^2 + y^{\ell} = 0
\end{equation}
for some $\ell \ge 2$;
\item
every $S\in B$ has at worst $m-1$ ADE singularities of types
\eqref{E000814} for $2\le \ell \le 4$ if \eqref{E000809} holds
and $I_\Lambda$ imposes independent conditions on $H^0(L)$ for every zero-dimensional
subscheme $\Lambda$ of $W$ supported at a single point $p$ and given by
\begin{equation}\label{E000813}
 \Lambda \cong \mO_W/\langle x_ix_j, x_i y^3, y^5\rangle
\end{equation}
for a set of generators $x_1, x_2, ..., x_{n-1}, y$ of $I_p$.
\end{enumerate}
\end{lemma}

\begin{proof}
These statements are again proved by a simple dimension count as in the proof of \lemref{LEM000}.

For (1), we let $V\subset W\times |L|$ be the incidence correspondence consisting of
pairs $(p, S)$ satisfying $S\in \PP H^0(L\otimes \mO_W(-3p))$. A similar argument as
in the proof of \lemref{LEM000} shows that
\begin{equation}\label{E000810}
 \dim V = \dim |L| - \binom{n+2}{2} + n \ge
\dim (\pi_2(V)).
\end{equation}
It follows that no $S\in B$ has singularities of multiplicity $\ge 3$ if \eqref{E000808} holds.

For (2), we let $U$ be the variety parameterizing zero-dimensional subschemes $\Lambda$
of $W$ supported at a single point $p$ with ideal sheaf $I_\Lambda$ given by
\begin{equation}\label{E000811}
 I_\Lambda = \text{Sym}^2 A \oplus I_p^3
\end{equation}
where $A$ is a subspace of $I_p/I_p^2$ of dimension $n-2$. A dimension count shows that
$\dim U = n + 2(n-2)$ and each $I_\Lambda$ imposes independent conditions on $H^0(L)$
since $I_\Lambda \supset I_p^3$. Therefore,
\begin{equation}\label{E000812}
 \dim V = h^0(L\otimes I_\Lambda) - 1 + \dim U = \dim |L| - 4
\end{equation}
for $V\subset U\times |L|$
the incidence correspondence consisting of pairs $(\Lambda, S)$ satisfying
$S\in \PP H^0(L\otimes I_\Lambda)$. 
Hence every $S\in B$ has at worst isolated double points of rank $\ge n-1$
if \eqref{E000809} holds.

For (3), we let $U$ be the variety parameterizing the zero-dimensional subschemes given by
\eqref{E000813}. It is easy to see that
\begin{equation}\label{E000815}
\dim U = 3n - 2 \text{ and }
h^0(L\otimes I_\Lambda) = h^0(L) - (3n+2)
\end{equation}
for all $\Lambda\in U$. If we let $V\subset U\times
|L|$ be the incidence correspondence consisting of pairs $(\Lambda, S)$ satisfying
$S\in \PP H^0(L\otimes I_\Lambda)$, then $\dim V = \dim |L| - 4$. Clearly, if $S\in |L|$ has
singularities of types \eqref{E000814} for $\ell \ge 5$, then
$S\in \pi_2(V)$. Therefore, every $S\in B$ has at worst ADE singularities of type
$x_1^2 + x_2^2 + ... + x_{n-1}^2 + y^{\ell} = 0$ for $2\le \ell \le 4$ if \eqref{E000809} holds.
\end{proof}

\begin{proof}[Proof of \propref{PROP802}]
By Mayer-Vietoris sequence \eqref{eq:bloch-1} and \eqref{E000800}, we see that
$H^4(X, \BQ)$ carries a pure Hodge structure of type $(2,2)$. So for every morphism
$f: X \to W$ from $X$ to a smooth projective variety $W$, we have
\begin{equation}\label{E000801}
f^* H^4(W, \BQ) = f^* H^{2,2}(W, \BQ).
\end{equation}
If we assume that the Hodge conjecture holds in codimension $2$, then \eqref{E000802} follows immediately.
Without the Hodge conjecture, we need to show instead that 
there exists $\xi\in \Ch^2(W; \BQ)$ for every $\omega\in H^{2,2}(W, \BQ)$ such that
$f^*(\cl_2(\xi) - \omega) = 0$ in $H^4(X, \BQ)$.

We may assume that $f: X\to W$ is an embedding. Otherwise, since $X$ is projective, we have an embedding
$g: X\hookrightarrow \PP^N$. Clearly, $f$ factors through $f\times g: X\hookrightarrow W\times \PP^N$ and
we may replace $f$ by $f\times g$. Furthermore, by cutting $W$ with sufficiently ample divisors passing
through $f(X)$ and the weak Lefschetz theorem, we may assume that $\dim W = 4$. 

We choose a sufficiently ample line bundle $L$ on $W$ such that
there is a member $X\cup X'$ in $|L|$ with the
property that $X\cup X'$ is a divisor with simple normal crossings.

Let $G = |L|$ and $M\subset G\times W$ be the universal family
$M = \{ (S, p): p\in S \}$ over $G$. Clearly, $M$ is smooth for $L$ sufficiently ample. And there
is a point $o\in G$ such that $M_o\cong X\cup X'$ for the fiber $M_o$ of $M/G$ over $o$.

Furthermore, $M_B = M\times_G B$ is smooth for a general linear subspace
$B\cong \PP^3\subset G$ passing through $o$.
By \lemref{LEM001} and by choosing $L$ sufficiently ample,
we see that a singular fiber $M_b$ of $M_B/B$ has only isolated
ADE singularities of types
\begin{equation}\label{E000803}
w^2 + x^2 + y^2 + z^m = 0\ (2\le m\le 4)
\end{equation}
for $b\ne o$.
For such $M_b$, we see that $H^2(M_b, \BQ)$ and $H^4(M_b, \BQ)$ carry pure
Hodge structures and the Hard Lefschetz theorem holds.

The map $f: X\hookrightarrow W$ clearly factors through $M_B$. So we may replace $W$ by $M_B$.
Finally, we have
\begin{itemize}
 \item $W$ is a smooth projective 6-fold flat over $B\cong \PP^3$ via $\rho: W \to B$;
\item $W_o = X\cup X'$ has simple normal crossings for a point $o\in B$;
\item $W_b$ has at worst singularities of types \eqref{E000803} for $b\in U= B\backslash \{o\}$;
\item $h^2(W_b, \BQ)$ and $h^{1,1}(W_b, \BQ)$ are constant for $b\in U$;
\item $H^2(W_b, \BQ)$ and $H^4(W_b, \BQ)$ carry pure Hodge structures with the Hard Lefschetz theorem
\begin{equation}\label{E000799}
L: H^{2}(W_b, \BQ) \xrightarrow[\wedge c_1(L)]{\cong} H^{4}(W_b,\BQ)
\end{equation}
for $b\in U$.
\end{itemize}

Let $\omega \in H^{2,2}(W, \BQ)$. Since the Hard Lefschetz theorem holds on all fibers of $W$ over
$U$, we can find $\xi\in \Ch^2(W; \BQ)$ such that $\cl_2(\xi_b) - \omega_b = 0$
in $H^4(W_b, \BQ)$ for all $b\in U$ using a Hilbert scheme argument. We claim that
$\cl_2(\xi_o) - \omega_o = 0$ in $H^4(W_o, \BQ)$.

Let $W_U = W\times_B U$. From the exact sequence
\begin{equation}\label{E000797}
... \xrightarrow{} 
H^7(W_o, \BQ)
\xrightarrow{}
H_c^{8}(W_U, \BQ) \xrightarrow{} H^{8}(W, \BQ)\xrightarrow{} H^8(W_o, \BQ)
\xrightarrow{} ...
\end{equation}
we see that $H^8(W, \BQ)\cong H_c^8(W_U, \BQ)$ and hence $H^4(W, \BQ) \cong H^4(W_U, \BQ)$.

The cohomologies $H^\bullet(W, \BQ)$ and $H^\bullet(W_U, \BQ)$ are computed by Leray spectral sequences
$E_r^{p,q}(W)$ and $E_r^{p,q}(W_U)$ whose $E_2$ terms are
\begin{equation}\label{E000798}
E_2^{p,q}(W) = H^p(B, R^q\rho_* \BQ)
\text{ and }
E_2^{p,q}(W_U) = H^p(U, R^q\rho_* \BQ)
\end{equation}
respectively (cf. \cite{gh}). Hence we have injections 
$E_r^{0,4}(W) \hookrightarrow H^0(B, R^4\rho_* \BQ)$ and 
$E_r^{0,4}(W_U) \hookrightarrow H^0(U, R^4\rho_* \BQ)$ for $r\ge 2$.

By the commutative diagram
\begin{equation}\label{E000804}
\xymatrix{
H^4(W, \BQ) \ar[r]^{\cong} \ar[d] & H^4(W_U, \BQ) \ar[d]\\
E_\infty^{0,4}(W)\ar[r]^{\cong} \ar[d]^\subset & E_\infty^{0,4}(W_U) \ar[d]^\subset\\
H^0(B, R^4\rho_* \BQ) \ar[r] & H^0(U, R^4\rho_* \BQ)
}
\end{equation}
we see that $\cl_2(\xi) - \omega$ vanishes in $H^0(B, R^4\rho_* \BQ)$ since it vanishes
in $H^0(U, R^4\rho_* \BQ)$. It follows that $\cl_2(\xi_o) - \omega_o$ vanishes in
$H^4(W_o, \BQ)$ and $f^*(\cl_2(\xi) - \omega) = 0$. We are done.
\end{proof}

 \end{document}